\documentclass[leqno,12pt]{article} %leqno is the option to put formula numbers on the left side
\usepackage{amsmath, amssymb}
\usepackage{amsthm} %theorem environment option
\theoremstyle{plain} %text of this environment is typesetted in italics
\newtheorem{theorem}{\indent\sc Theorem}[section]

\newtheorem{corollary}[theorem]{\indent\sc Corollary}
\newtheorem{proposition}[theorem]{\indent\sc Proposition}

\theoremstyle{definition} %text of this environment is typesetted in roman letters
\newtheorem{definition}[theorem]{\indent\sc Definition}
\newtheorem{remark}[theorem]{\indent\sc Remark}
\newtheorem{example}[theorem]{\indent\sc Example}

\title{A note on warped product\\
almost quasi-Yamabe solitons} %title of the paper

\author{Adara M. Blaga}
\date{} %leave empty
\begin{document}

\maketitle

%\centerline{\textit{\small{Dedicated to Professor Bang-Yen Chen on his 75th anniversary}}}
%%%%%%%%%%%%%%% footnote %%%%%%%%%%%%%%%%
\footnote{ %2010 MSC numbers
}
\footnote{ %key words and phrases
}
%\footnote{ %acknowledgment of support etc. if any
%$^{*}$The author acknowledges the support by the research grant PN-II-ID-PCE-2011-3-0921.
%}
%%%%%%%%%%%%%%%%%%%%%%%%%%%%%%%%%%%%%%%%%

\begin{abstract}
We consider almost quasi-Yamabe solitons in Riemannian manifolds, derive a Bochner-type formula in the gradient case and prove that under certain assumptions, the manifold is of constant scalar curvature. We also provide necessary and sufficient conditions for a gradient almost quasi-Yamabe soliton on the base manifold to induce a gradient almost quasi-Yamabe soliton on the warped product manifold.
\end{abstract}

\section{Introduction}

The notion of \textit{Yamabe solitons}, which generate self-similar solutions to Yamabe flow \cite{ham}:
\begin{equation}
\frac{\partial }{\partial t}g(t)=-scal(t) \cdot g(t),
\end{equation}
firstly appeared to L. F. di Cerbo and M. N. Disconzi in \cite{c}. In \cite{chen}, B.-Y. Chen introduced the notion of \textit{quasi-Yamabe soliton} which we shall consider in the present paper for a more general case, when the constants are let to be functions.

Let $(M,g)$ be an $n$-dimensional Riemannian manifold ($n>2$), $\xi$ a vector field and $\eta$ a $1$-form on $M$.
\begin{definition}
\textit{An almost quasi-Yamabe soliton} on $M$ is a data $(g,\xi,\lambda,\mu)$ which satisfy the equation:
\begin{equation}\label{e8}
\frac{1}{2}\mathcal{L}_{\xi}g+(\lambda-scal) g+\mu\eta\otimes \eta=0,
\end{equation}
where $\mathcal{L}_{\xi}$ is the Lie derivative operator along the vector field $\xi$ and $\lambda$ and $\mu$ are smooth functions on $M$.
\end{definition}

When the potential vector field of (\ref{e8}) is of gradient type, i.e. $\xi=grad(f)$, then $(g,\xi,\lambda,\mu)$ is said to be a \textit{gradient almost quasi-Yamabe soliton} (or a generalized quasi-Yamabe gradient soliton) \cite{cc} and the equation satisfied by it becomes:
\begin{equation}\label{e22}
Hess(f)+(\lambda-scal) g+\mu df\otimes df=0.
\end{equation}

\bigskip

In the next section, we shall derive a Bochner-type formula for the gradient almost quasi-Yamabe soliton case and
prove that under certain assumptions, the manifold is of constant scalar curvature. In the last section we construct an almost quasi-Yamabe soliton on a warped product manifold. Remark that results on warped product gradient Yamabe solitons for certain types of warping functions $f$ have been obtained by W. I. Tokura, L. R. Adriano and R. S. Pina in \cite{t}.

\section{Gradient almost quasi-Yamabe solitons}

Remark that in the gradient case, from (\ref{e22}) we get:
\begin{equation}\label{p}
\nabla\xi=-(\lambda-scal)I-\mu df\otimes \xi.
\end{equation}

Therefore, $\nabla_{\xi}\xi=[\Delta(f)+(n-1)(\lambda-scal)]\xi$, i.e. $\xi$ is a
\textit{generalized geodesic vector field} with the potential function $\Delta(f)+(n-1)(\lambda-scal)$ \cite{chen2}.

Also, if $(\lambda,\mu)=(scal-1,1)$, then $\xi$ is \textit{torse-forming} and if $\mu=0$, then $\xi$ is \textit{concircular}.

\bigskip

Now we shall get a condition that $\mu$ should satisfy in a gradient almost quasi-Yamabe soliton $(g,\xi,\lambda,\mu)$. Taking the scalar product with $Hess(f)$, from (\ref{e22}) we get:
$$|Hess(f)|^2+(\lambda-scal)\Delta(f)+\frac{\mu}{2}\xi(|\xi|^2)=0$$
and tracing (\ref{e22}) we obtain:
$$\Delta(f)+n(\lambda-scal)+\mu |\xi|^2=0.$$

From the above relations we deduce the equation:
$$n\lambda^2+(2n\cdot scal+\mu |\xi|^2)\lambda+n\cdot scal^2+\mu |\xi|^2\cdot scal-\frac{\mu}{2}\xi(|\xi|^2)-|Hess(f)|^2=0 $$
which has solution (in $\lambda$) if and only if $$\mu^2|\xi|^4+2n\mu\xi(|\xi|^2)+4n|Hess(f)|^2\geq 0$$
(that is always true for $\xi$ of constant length).

\bigskip

The next step is to deduce a Bochner-type formula for the gradient almost quasi-Yamabe soliton case.
\begin{theorem}\label{t}
If (\ref{e22}) defines a gradient almost quasi-Yamabe soliton on the $n$-dimen\-sio\-nal Riemannian manifold $(M,g)$ and
$\eta=df$ is the $g$-dual of the gradient vector field $\xi:=grad(f)$, then:
\begin{equation}\label{e53}
\frac{1}{2}\Delta(|\xi|^2)=|\nabla \xi|^2-\frac{1}{n-1}S(\xi,\xi)-\frac{n-2}{2(n-1)}\mu \nabla_{\xi}(|\xi|^2)-\end{equation}$$-|\xi|^2[\xi(\mu)-\frac{n}{n-1}\mu^2|\xi|^2-\frac{n^2}{n-1}\lambda \mu+\frac{n^2}{n-1}\mu \cdot scal].
$$
\end{theorem}
\begin{proof}
First remark that:
$$
trace(\mu \eta\otimes \eta)=\mu|\xi|^2
$$
and
$$div(\mu \eta\otimes \eta)=\frac{\mu}{2}d(|\xi|^2)+\mu \Delta(f) df+d\mu(\xi)df.$$

Taking the trace of the equation (\ref{e22}), we obtain:
\begin{equation}\label{e13}
\Delta(f)+n(\lambda-scal) +\mu |\xi|^2=0
\end{equation}
and differentiating it:
\begin{equation}\label{e14}
d(\Delta(f))+nd\lambda-nd(scal)+\mu d(|\xi|^2)+|\xi|^2d\mu=0.
\end{equation}

Now taking the divergence of the same equation, we get:
\begin{equation}\label{e15}
div(Hess(f))+d\lambda-d(scal)+\frac{\mu}{2}d(|\xi|^2)+\mu \Delta(f) df+d\mu(\xi)df=0.
\end{equation}

Substracting the relations (\ref{e15}) and (\ref{e14}) computed in $\xi$ and using \cite{bla}:
$$div(Hess(f))=d(\Delta(f))+i_{Q\xi}g,$$
$$(div(Hess(f)))(\xi)=\frac{1}{2}\Delta(|\xi|^2)-|\nabla \xi|^2,$$
we obtain (\ref{e53}).
\end{proof}

\begin{remark}
For the case $\mu=0$, under the assumptions $S(\xi,\xi)\leq (n-1)|\nabla \xi|^2$ we get $\Delta(|\xi|^2)\geq 0$ and from the maximum principle follows that $|\xi|^2$ is constant in a neighborhood of any local maximum. If $|\xi|$ achieve its maximum, then $S(\xi,\xi)= (n-1)|\nabla \xi|^2$.
\end{remark}

Let us make some remarks on the scalar curvature of $M$.

From (\ref{p}) we get:
$$R(\cdot,\cdot)\xi=-[d(\lambda-scal)\otimes I-I\otimes d(\lambda-scal)]-\mu(\lambda-scal)(df\otimes I-I\otimes df)-(d\mu \otimes df-df\otimes d\mu)$$
and
\begin{equation}\label{pp}
R(\cdot,\xi)\cdot=d(\lambda-scal)\otimes I-g\otimes[grad(\lambda-scal)-\mu (\lambda-scal)\xi]+\mu(\lambda-scal)df\otimes I+
\end{equation}
$$+d\mu\otimes df\otimes \xi
-df\otimes df\otimes grad(\mu)$$
which for $\lambda$ and $\mu$ constant become:
$$R(\cdot,\cdot)\xi=[d(scal)\otimes I-I\otimes d(scal)]-\mu(\lambda-scal)(df\otimes I-I\otimes df)$$
and
\begin{equation}\label{pg}
R(\cdot,\xi)\cdot=-d(scal)\otimes I+g\otimes[grad(scal)+\mu (\lambda-scal)\xi]+\mu(\lambda-scal)df\otimes I.
\end{equation}

Using (\ref{pp}), $R(\xi,\xi)X=0$ implies:
$$[d(\lambda-scal)+|\xi|^2d\mu]\otimes \xi=df \otimes [grad(\lambda-scal)+|\xi|^2 grad(\mu)]$$
which for $\lambda$ and $\mu$ constant becomes:
$$d(scal)\otimes \xi=df \otimes grad(scal).$$

Assume further that $\lambda$ and $\mu$ are constant.
Computing the previous relation in $\xi$ and choosing an open subset where $\xi\neq 0$, we deduce:
\begin{equation}\label{po}
grad(scal)=\frac{\xi(scal)}{|\xi|^2}\xi.
\end{equation}

Denoting by $h=:\frac{\displaystyle \xi(scal)}{\displaystyle |\xi|^2}$, from the symmetry of $Hess(scal)$ we obtain:
$$dh \otimes df=df\otimes dh$$
which implies:
$$|\xi|^2dh=\xi(h) df \ \ \textit{and} \ \ |\xi|^2 grad(h)=\xi(h)\xi.$$

A similar result like the one obtained by B.-Y. Chen, S. Deshmukh in \cite{chen2} for Yamabe solitons can be obtained for quasi-Yamabe solitons, following the same ideas in proving it.
\begin{theorem}
Let (\ref{e22}) define a gradient quasi-Yamabe soliton on the connected $n$-dimen\-sio\-nal Riemannian manifold $(M,g)$ ($n>1$) for
$\eta=df$ the $g$-dual of the unitary vector field $\xi:=grad(f)$. If $\xi(scal)$ is constant along the integral curves of $\xi$ and $Hess(scal)$ is degenerate in the direction of $\xi$, then $M$ is of constant scalar curvature.
\end{theorem}
\begin{proof}
Under these hypotheses, applying divergence to (\ref{po}) we obtain:
\begin{equation}\label{l}
\Delta(scal)=\xi(scal) \Delta(f)=-[n(\lambda-scal)+\mu]\xi(scal).
\end{equation}

Computing the Ricci operator in $\xi$, $Q\xi=-\sum_{i=1}^nR(E_i,\xi)E_i$, for $\{E_i\}_{1\leq i\leq n}$ a local orthonormal frame field on $M$, and using (\ref{pg}) we get:
\begin{equation}\label{nn}
Q\xi=-(n-1)grad(scal)+(n-1)\mu(\lambda-scal)\xi
\end{equation}
and
\begin{equation}\label{j}
S(\xi,\xi)=g(Q\xi,\xi)=-(n-1)\xi(scal)+(n-1)\mu(\lambda-scal).
\end{equation}

Applying the divergence to (\ref{nn}) we have:
\begin{equation}
div(Q\xi)=-(n-1)\Delta(scal)+(n-1)\mu[(\lambda-scal)\Delta(f)-\xi(scal)].
\end{equation}

Computing the same divergence like:
\begin{equation}
div(Q\xi)=div(S)(\xi)+\langle S, Hess(f)\rangle,
\end{equation}
taking into account the gradient quasi-Yamabe soliton equation, the fact that $$div(S)(\xi)=\frac{\displaystyle \xi(scal)}{\displaystyle 2},$$ the expression of $S(\xi,\xi)$ from (\ref{j}) and replacing $\Delta(scal)$ from (\ref{l}), we obtain:
$$\left[\frac{1}{2}-n(n-1)(\lambda-scal)+(n-1)\mu\right]\xi(scal)=$$$$=(\lambda-scal)[(1+n(n-1)\mu)scal-n(n-1)\lambda\mu].$$

Differentiating the previous expression along $\xi$ and taking into account the degeneracy of $Hess(scal)(\xi,\xi)=\xi(\xi(scal))-(\nabla_{\xi}\xi)(scal)$
in the direction of $\xi$, after a long computation, we get:
$$\xi(scal)\left[\xi(scal)+scal^2+k_1scal+k_2\right]=0,$$
where the constants $k_1$ and $k_2$ are respectively given by:
$$k_1=:\frac{n+1}{n}\mu+\frac{5}{2n(n-1)}, \ \ k_2=:\lambda^2-\frac{1}{n}\mu^2-\frac{n+1}{n}\lambda\mu-\frac{3\lambda+\mu}{2n(n-1)}.$$

Differentiating again the term in the parantheses along $\xi$ we get:
$$\xi(scal)\left[3scal-\lambda+\frac{1}{n}\mu+\frac{5}{2n(n-1)}\right]=0$$
which completes the proof.
\end{proof}

\section{Warped product almost quasi-Yamabe solitons}

\subsection{Warped product manifolds}

Consider $(B,g_B)$ and $(F,g_F)$ two Riemannian manifolds of dimensions $n$ and $m$, respectively. Denote by $\pi$ and $\sigma$ the projection maps from the product manifold $B\times F$ to $B$ and $F$ and by $\widetilde{\varphi}:=\varphi \circ \pi$ the lift to $B\times F$ of a smooth function $\varphi$ on $B$.
In this context, we shall call $B$ \textit{the base} and $F$ \textit{the fiber} of $B\times F$, the unique element $\widetilde{X}$ of $\chi(B\times F)$ that is $\pi$-related to $X\in \chi(B)$ and to the zero vector field on $F$, the \textit{horizontal lift of $X$} and the unique element $\widetilde{V}$ of $\chi(B\times F)$ that is $\sigma$-related to $V\in \chi(F)$ and to the zero vector field on $B$, the \textit{vertical lift of $V$}.
Also denote by $\mathcal{L}(B)$ the set of all horizontal lifts of vector fields on $B$, by $\mathcal{L}(F)$ the set of all vertical lifts of vector fields on $F$, by $\mathcal{H}$ the orthogonal projection of $T_{(p,q)}(B\times F)$ onto its horizontal subspace $T_{(p,q)}(B\times \{q\})$ and by
$\mathcal{V}$ the orthogonal projection of $T_{(p,q)}(B\times F)$ onto its vertical subspace $T_{(p,q)}(\{p\}\times F)$.

Let $\varphi>0$ be a smooth function on $B$ and
\begin{equation}\label{e7}
g:=\pi^* g_B+(\varphi\circ \pi)^2\sigma^*g_F
\end{equation}
be a Riemannian metric on $B\times F$.
\begin{definition}\cite{bi}
The product manifold of $B$ and $F$ together with the Riemannian metric $g$ defined by (\ref{e7}) is called \textit{the warped product} of $B$ and $F$ by the warping function $\varphi$ (and is denoted by $(M:=B\times_{\varphi} F,g)$).
\end{definition}

In particular, if $\varphi=1$, then the warped product becomes the usual product of the Riemannian manifolds.

\bigskip

For simplification, in the rest of the paper we shall simply denote by $X$ the horizontal lift of $X\in \chi(B)$ and by $V$ the vertical lift of $V\in \chi(F)$.

\bigskip

Notice that the lift on $M$ of the gradient and the Hessian satisfy:
\begin{equation}
grad(\widetilde{f})=\widetilde{grad(f)},
\end{equation}
\begin{equation}
(Hess(\widetilde{f}))(X,Y)=\widetilde{(Hess(f))(X,Y)}, \ \ \textit{for any } \ X, Y \in \mathcal{L}(B),
\end{equation}
for any smooth function $f$ on $B$.

Also, the scalar curvatures are connected by the relation \cite{Fernando Dobarro}:
\begin{equation}\label{e33}
scal=\widetilde{scal}_B+\frac{\widetilde{scal}_F}{\varphi^2}-\pi^*\left(2m \frac{\Delta(\varphi)}{\varphi}+m(m-1)\frac{|grad(\varphi)|^2}{\varphi^2}\right).
\end{equation}

\subsection{Warped product almost quasi-Yamabe solitons}

We shall construct a gradient almost quasi-Yamabe soliton on a warped product manifold.

Let $(B,g_B)$ be an $n$-dimensional Riemannian manifold, $\varphi>0$ a smooth function on $B$ and $f$, $\mu$ smooth functions on $B$ such that:
\begin{equation}\label{e36}
\Delta(f)+\mu|grad(f)|^2=n\frac{(grad(f))(\varphi)}{\varphi}.
\end{equation}

In this case, any gradient almost quasi-Yamabe soliton $(g_B,grad(f),\lambda_B,\mu_B)$ on $(B,g_B)$ is given by $\lambda_B=scal_B-\frac{\displaystyle (grad(f))(\varphi)}{\displaystyle \varphi}$ and $\mu_B=\mu$.

Take $(F,g_F)$ an $m$-dimensional manifold with
\begin{equation}\label{e44}
scal_F=\pi^*((\lambda-\lambda_B)\varphi^2+2m\varphi\Delta(\varphi)+m(m-1)|grad(\varphi)|^2)|_F,
\end{equation}
where $\pi$ and $\sigma$ are the projection maps from the product manifold $B\times F$ to $B$ and $F$, respectively, $g:=\pi^* g_B+(\varphi\circ \pi)^2\sigma^*g_F$ is a Riemannian metric on $B\times F$, $\lambda_B=scal_B-\frac{\displaystyle (grad(f))(\varphi)}{\displaystyle \varphi}$ and $\lambda$ is a smooth function on $B$.

With the above notations, we prove:
\begin{theorem}
Let $(B,g_B)$ be an $n$-dimensional Riemannian manifold, $\varphi>0$, $f$, $\mu$ smooth functions on $B$ satisfying (\ref{e36}) and $(F,g_F)$ an $m$-dimensional Riemannian manifold with the scalar curvature given by (\ref{e44}). Then $(g,\xi,\pi^*(\lambda),\pi^*(\mu))$, where $\xi=grad(\widetilde{f})$, is a gradient almost quasi-Yamabe soliton on the warped product manifold $(B\times_{\varphi} F,g)$
if and only if $(g_B,grad(f),\lambda_B=scal_B-\frac{\displaystyle (grad(f))(\varphi)}{\displaystyle \varphi},\mu)$ is a gradient almost quasi-Yamabe soliton on $(B,g_B)$.
\end{theorem}
\begin{proof}
The gradient almost quasi-Yamabe soliton $(g,\xi,\pi^*(\lambda),\pi^*(\mu))$ on $(B\times_{\varphi} F,g)$ is given by:
\begin{equation}\label{e31}
Hess(\widetilde{f})+(\pi^*(\lambda)-scal) g+\pi^*(\mu) d\widetilde{f}\otimes d\widetilde{f}=0.
\end{equation}
Notice that from (\ref{e33}), (\ref{e36}) and (\ref{e44}) we deduce that $$\pi^*(\lambda)-scal=\pi^*(\lambda_B)-\widetilde{scal}_B,$$ hence
for any $X$, $Y\in \mathcal{L}(B)$ we get:
\begin{equation}\label{eeee}
H^{f}(X,Y)+(\lambda_B-scal_B) g_B(X,Y)+\mu df(X)df(Y)=0
\end{equation}
i.e. $(g_B,grad(f),\lambda_B,\mu)$ is a gradient almost quasi-Yamabe soliton on $(B,g_B)$, where $H^{f}$ denotes the lift of $Hess(f)$.

Conversely, notice that the left-hand side term in (\ref{e31}) computed in $(X,V)$, for $X\in \mathcal{L}(B)$ and $V\in \mathcal{L}(F)$ vanishes identically and for each situation $(X,Y)$ and $(V,W)$, we can recover the equation (\ref{e31}) from (\ref{e36}) and the fact that $(g_B,grad(f),\lambda_B,\mu)$ is a gradient almost quasi-Yamabe soliton on $(B,g_B)$. Indeed, taking the trace of (\ref{eeee}) we get $$\Delta(f)+n(\lambda_B-scal_B)+\mu |grad(f)|^2=0$$
and using (\ref{e36}) we obtain $$\pi^*(\lambda_B)-\widetilde{scal}_B=-\frac{(grad(f))(\varphi)}{\varphi}.$$

We know that for any $V$, $W\in \mathcal{L}(F)$:
$$H^{f}(V,W)=(Hess(\widetilde{f}))(V,W)=g(\nabla_V(grad(\widetilde{f})),W)=$$$$
=\pi^*\left[\frac{(grad(f))(\varphi)}{\varphi}\right]|_F\widetilde{\varphi}^2|_Fg_F(V,W)$$
and we deduce that
$$H^{f}(V,W)+(\pi^*(\lambda_B)-\widetilde{scal}_B)|_Fg(V,W)=0.$$
\end{proof}

\begin{example}
Consider $M=\{(x,y,z)\in\mathbb{R}^3, z> 0\}$, where $(x,y,z)$ are the standard coordinates in $\mathbb{R}^3$,
$$g_M:=\frac{1}{z^2}(dx\otimes dx+dy\otimes dy+dz\otimes dz) \ \ \textit{and} \ \ \xi_M:=-z\frac{\partial}{\partial z}.$$

Let $(g_M, \xi_M, -8,2)$ be the gradient quasi-Yamabe soliton on the Riemannian manifold $(M,g_M)$ and let $S^3$ be the $3$-sphere with the round metric $g_S$ (which is Einstein with the Ricci tensor equals to $2g_S$). Thus we obtain the gradient quasi-Yamabe soliton $(g, \xi, -2,2)$ on the "generalized cylinder" $M\times S^3$, where $g=g_M+g_S$ and $\xi$ is the lift on $M\times S^3$ of the gradient vector field $\xi_M=grad(f)$, where $f(x,y,z):=-\ln z$.
%$scal_M=-6$
\end{example}

\subsection{Some consequences of condition (\ref{e36})}

Let us make some remarks on the class of manifolds that satisfy the condition:
\begin{equation}\label{e37}
\Delta(f)+\mu|\xi|^2=n\frac{d\varphi(\xi)}{\varphi},
\end{equation}
for $\varphi>0$, $f$ and $\mu$ smooth functions on the oriented and compact Riemannian manifold $(B,g_B)$ and $\xi:=grad(f)$.

Remark that if
\begin{equation}\label{rrr}
Hess(f)-\frac{n}{2\varphi}(df\otimes d\varphi+d\varphi\otimes df)+\mu df\otimes df=0,
\end{equation}
then (\ref{e37}) is satisfied. Computing $Hess(f)(X,Y):=g_B(\nabla_X\xi,Y)$ we get
$$\nabla \xi=\frac{n}{2\varphi}(df\otimes grad(\varphi)+d\varphi\otimes \xi)-\mu df\otimes \xi.$$

Also notice that in this case, if $(g_B,\xi,\lambda_B,\mu)$ is a gradient almost quasi-Yamabe soliton on $(B,g_B)$, then the metric $g_B$ is precisely $$g_B=-\frac{n}{2\varphi(\lambda_B-scal_B)}(df\otimes d\varphi+d\varphi\otimes df)$$
and $scal_B=\lambda_B+\frac{\displaystyle d\varphi(\xi)}{\displaystyle \varphi}$.

In what follows, we shall focus on condition (\ref{rrr}). We've checked that \cite{bla}:
$$
|Hess(f)-\frac{\Delta(f)}{n}g_B|^2=|Hess(f)|^2-\frac{(\Delta(f))^2}{n},
$$
$$(div(Hess(f)))(\xi)=div(Hess(f)(\xi))-|Hess(f)|^2,$$
therefore:
\begin{equation}\label{e40}
(div(Hess(f)))(\xi)=div(Hess(f)(\xi))-|Hess(f)-\frac{\Delta(f)}{n}g_B|^2-\frac{(\Delta(f))^2}{n}.
\end{equation}

Applying the divergence to (\ref{rrr}), computing it in $\xi$ and taking into account that
$$div\left(\frac{1}{\varphi}df\otimes d\varphi\right)=\left(\frac{\Delta(f)}{\varphi}-\frac{d\varphi(\xi)}{\varphi^2}\right) d\varphi+\frac{1}{\varphi}i_{\nabla_{\xi}grad(\varphi)}g_B$$
and
$$div(\mu df\otimes df)=\frac{\mu}{2}d(|\xi|^2)+\mu \Delta(f)df+d\mu(\xi)df,$$
we get:
\begin{equation}\label{e38}
(div(Hess(f)))(\xi)=n\left(\frac{\Delta(f)}{\varphi}-\frac{d\varphi(\xi)}{\varphi^2}\right) d\varphi(\xi)+\frac{n}{\varphi}g_B({\nabla_{\xi}grad(\varphi)},\xi)-\end{equation}$$-\frac{\mu}{2}d(|\xi|^2)(\xi)-\mu\Delta(f)|\xi|^2-d\mu(\xi)|\xi|^2
$$
and we obtain:
\begin{equation}\label{e51}
div(Hess(f)(\xi))=|Hess(f)-\frac{\Delta(f)}{n}g_B|^2+\frac{(\Delta(f))^2}{n}+n\left(\frac{\Delta(f)}{\varphi}-\frac{d\varphi(\xi)}{\varphi^2}\right) d\varphi(\xi)+\end{equation}$$+\frac{n}{\varphi}g_B({\nabla_{\xi}grad(\varphi)},\xi)-\frac{\mu}{2}d(|\xi|^2)(\xi)-\mu\Delta(f)|\xi|^2-d\mu(\xi)|\xi|^2.
$$

Integrating with respect to the canonical measure on $B$, we have:
$$\int_Bd(|\xi|^2)(\xi)=\int_B\langle grad(|\xi|^2), \xi\rangle=-\int_B\langle |\xi|^2, div(\xi)\rangle=-\int_B|\xi|^2 \cdot \Delta(f).$$

Using:
$$|\xi|^2\cdot \Delta(f)=|\xi|^2\cdot div(\xi)=div(|\xi|^2\xi)-|\xi|^2$$
taking $\mu$ constant and integrating (\ref{e51}) on $B$, from the above relations and the divergence theorem, we obtain:
\begin{equation}\label{e56}
\int_B |Hess(f)-\frac{\Delta(f)}{n}g_B|^2+(n+1)\int_B \Delta(f)\cdot \frac{d\varphi(\xi)}{\varphi}+\frac{(2-\mu)n+2}{2n}\int_B|\xi|^2-
\end{equation}
$$-n\int_B \frac{(d\varphi(\xi))^2}{\varphi^2}+n\int_B g_B(\frac{1}{\varphi}\nabla_{\xi}grad(\varphi),\xi)=0.$$

Assume now that $\mu$ is constant and consider the product manifold $B\times F$, in which case (\ref{rrr}) and (\ref{e37}) (for $\varphi=1$) become:
\begin{equation}\label{er}
Hess(f)+\mu df\otimes df=0  \ \ \textit{and} \ \ \Delta(f)+\mu |\xi|^2=0.
\end{equation}

\begin{remark}\label{r1}
i) In the case of product manifold (for $\varphi =1$), the chosen manifold $(F, g_F)$ is of scalar curvature $scal_F=\pi^*(\lambda-scal_B)|_F$. In particular, for $\lambda=scal_B$, $(F,g_F)$ is locally isometric to an Euclidean space. Moreover, $\nabla_{\xi}\xi=-\mu |\xi|^2\xi$, therefore, $\xi$ is a generalized geodesic vector field with the potential function $\Delta(f)$.

ii) For $\varphi=1$ and $\mu$ constant, we obtain:
$$\mu^2\int_B|\xi|^4=0$$
\end{remark}
and we can state:
\begin{corollary}
Let $(B,g_B)$ be an oriented and compact $n$-dimensional Riemannian manifold, $f$ a smooth function on $B$ and $\mu$ a real constant satisfying (\ref{er}). Then $\mu=0$ hence, $f$ is harmonic and $\nabla \xi=0$.
\end{corollary}

\begin{proposition}
Let $(B,g_B)$ be an oriented, compact and complete $n$-dimensional ($n>1$) Riemannian manifold, $f$ a smooth function on $B$ and $\mu$ a real constant satisfying (\ref{er}). Then $B$ is conformal to a sphere in the $(n+1)$-dimensional Euclidean space.
\end{proposition}
\begin{proof}
From the above observations, we have:
$$\int_B |Hess(f)-\frac{\Delta(f)}{n}g_B|^2=\int_B|Hess(f)|^2-\int_B\frac{(\Delta(f))^2}{n}=\frac{n-1}{n}\mu^2\int_B|\xi|^4=0,$$
so $Hess(f)=\frac{\displaystyle \Delta(f)}{\displaystyle n}g_B$ which implies by \cite{ya} that $B$ is conformal to a sphere in the $(n+1)$-dimensional Euclidean space.
\end{proof}

\small{

\bigskip

\textit{Adara M. Blaga}

\textit{Department of Mathematics}

\textit{West University of Timi\c{s}oara}

\textit{Bld. V. P\^{a}rvan nr. 4, 300223, Timi\c{s}oara, Rom\^{a}nia}

\textit{adarablaga@yahoo.com}
}

\begin{thebibliography}{99}
\bibitem{bi} {R. L. Bishop, B. O'Neill}, \textit{Manifolds of negative curvature}, Trans. Amer. Math. Soc. \textbf{145}, 1-49 (1969).

\bibitem{b12}Blaga, A. M.: \emph{A note on almost $\eta$-Ricci solitons in Euclidean hypersurfaces}. Serdica Math. J. \textbf{43}(3-4), 361-368 (2017).

\bibitem{b}Blaga, A. M.: \emph{Almost $\eta$-Ricci solitons in $(LCS)_n$-manifolds}. Bull. Belg. Math. Soc. Simon Stevin \textbf{25}(5), 641-653 (2018).

\bibitem{blahh}Blaga, A. M.: \emph{$\eta$-Ricci solitons on Lorentzian para-Sasakian manifolds}. Filomat \textbf{30}(2), 489-496 (2016).

\bibitem{bl}Blaga, A. M.: \emph{$\eta$-Ricci solitons on para-Kenmotsu manifolds}. Balkan J. Geom. Appl. \textbf{20}(1), 1-13 (2015).

\bibitem{b11}Blaga, A. M.: \emph{Last multipliers on $\eta$-Ricci solitons}. Matematichki Bilten \textbf{42}(2), 85-90 (2018).

\bibitem{blaga}Blaga, A. M.: \emph{On gradient $\eta$-Einstein solitons}. Kragujevak J. Math. \textbf{42}(2), 229-237 (2018).

%\bibitem{ada}Blaga, A. M.: \emph{On harmonicity and Miao-Tam critical metrics in a perfect fluid spacetime}. Bol. Soc. Mat. Mexicana, https://doi.org/10.1007/s40590-020-00281-4 (2020).

\bibitem{b3}Blaga, A. M.: \emph{On solitons in statistical geometry}. Int. J. Appl. Math. Stat. \textbf{58}(4), (2019).

\bibitem{bla}Blaga, A. M.: \emph{On warped product gradient $\eta$-Ricci solitons}. Filomat \textbf{31}(18), 5791-5801 (2017).

\bibitem{b7}Blaga, A. M.: \emph{Remarks on almost $\eta$-solitons}. Matematicki Vesnik \textbf{71}(3), 244-249 (2019).

%\bibitem{blagam}Blaga, A. M.: \emph{Solitons and geometrical structures in a perfect fluid spacetime}. Rocky Mountain J. Math. \textbf{50}(1), 41-53 2020.

\bibitem{b1}Blaga, A. M.: \emph{Solutions of some types of soliton equations in $\mathbb{R}^3$}. Filomat \textbf{33}(4), 1159-1162 (2019).

\bibitem{b5}Blaga, A. M.: \emph{Some geometrical aspects of Einstein, Ricci and Yamabe solitons}. J. Geom. Sym. Phys. \textbf{52}, 17-26 (2019).

%\bibitem{b6}{Blaga, A. M., Baishya, K. K., Sarkar, N.}: \emph{Ricci solitons in a generalized weakly (Ricci) symmetric $D$-homothetically deformed Kenmotsu manifold}. Ann. Univ. Paedagog. Crac. Stud. Math. \textbf{18}, 123-136 (2019).

%\bibitem{b2}{Blaga, A. M., Crasmareanu, M. C.}: \emph{Inequalities for gradient Einstein and Ricci solitons}. Facta Univ. Math. Inform. \textbf{35}(2), 351-356 (2020).

\bibitem{blcr}{Blaga, A. M., Crasmareanu, M. C.}: \emph{Torse-forming $\eta$-Ricci solitons in almost paracontact $\eta$-Einstein geometry}. Filomat \textbf{31}(2), 499-504 (2017).

\bibitem{b8}{Blaga, A. M., Perkta\c s, S. Y.}: \emph{Remarks on almost $\eta$-Ricci solitons in $\varepsilon$-para Sasakian manifolds}. Commun. Fac. Sci. Univ. Ank. Ser. A1 Math. Stat. \textbf{68}(2), 1621-1628 (2019).

\bibitem{b9}{Blaga, A. M., Perkta\c s, S. Y., Acet, B. E., Erdogan, F. E.}: \emph{$\eta$-Ricci solitons in $\varepsilon$-almost paracontact metric manifolds}. Glasnik Matematicki \textbf{53}(1), 377-410 (2018).





\bibitem{chen} Chen, B.-Y., Deshmukh, S.: \textit{Yamabe and quasi-Yamabe solitons on Euclidean submanifolds}, arXiv:1711.02978, 2017.

%\bibitem{b10}{De, K., Blaga, A. M., De, U. C.}: \emph{$*$-Ricci solitons on ($\varepsilon$)-Kenmotsu manifolds}. Palestine Math. J. \textbf{9}(2), 984-990 (2020).



\bibitem{chen2} Deshmukh, S., Chen, B.-Y.: \textit{A note on Yamabe solitons}, https://www.researchgate.net/publication/323943715.

\bibitem{c} Di Cerbo, L. F., Disconzi, M. N.: \textit{Yamabe Solitons, Determinant of the Laplacian and the Uniformization Theorem for Riemann Surfaces}, Lett. Math. Phys. \textbf{83}(1), 13-18 (2008).
\bibitem{Fernando Dobarro} Dobarro, F., \"{U}nal, B.: \textit{Curvature of multiply warped products}, J. Geom. Phys. \textbf{55}(1), 75-106 (2005).
\bibitem{ham} {Hamilton, R. S.}: \textit{The Ricci flow on surfaces, Math. and general relativity} (Santa Cruz, CA, 1986), 237-262, Contemp. Math. \textbf{71} (1988), AMS.
\bibitem{cc} Neto, B. L., De Oliveira, H. P.: \textit{Generalized quasi Yamabe gradient solitons}, Diff. Geom. Appl. \textbf{49}, 167-175 (2016).
\bibitem{t} Tokura, W. I., Adriano, L. R., Pina, R. S.: \textit{On warped product gradient Yamabe soliton}, arXiv:1711.11455, 2017.
\bibitem{ya} {Yano, K., Obata, M.}: \textit{Conformal changes of Riemannian metrics}, J. Differential Geom. \textbf{4}, 53-72 (1970).
\end{thebibliography}
\end{document}